\documentclass{ipi}
\usepackage{
amsmath, amsthm, amssymb, graphicx, commath, pgfplots, pst-solides3d, inputenc, float}
\newcommand{\no}[1]{#1}

\renewcommand{\no}[1]{}  \newcommand{\upDelta}{\Delta} 

\no{\usepackage{times}\usepackage[
subscriptcorrection, slantedGreek,
nofontinfo]{mtpro}
\renewcommand{\Delta}{\upDelta}
}

 \usepackage[colorlinks=true]{hyperref}
\hypersetup{urlcolor=blue, citecolor=red}

\def\currentvolume{12}
 \def\currentissue{2}
  \def\currentyear{2018}
   
    \def\ppages{X--XX}
     \def\DOI{10.3934/ipi.2018013}

\numberwithin{equation}{section}%

\newtheorem{theorem}{Theorem}[section]
\newtheorem{proposition}{Proposition}[section]
\newtheorem{lemma}{Lemma}[section]

\newtheorem{corollary}{Corollary}[section]

{\theoremstyle{definition}
\newtheorem{definition}{Definition}[section]
\newtheorem{remark}{Remark}[section]
\newtheorem{example}{Example}[section]}

\DeclareMathOperator{\supp}{supp}

\DeclareMathOperator{\WF}{WF}
\DeclareMathOperator{\WFA}{WF_A}
\DeclareMathOperator{\curl}{curl}

\DeclareMathOperator{\ext}{ext}
\DeclareMathOperator{\const}{const}

\newcommand{\eps}{\varepsilon}
\newcommand{\R}{{\bf R}}
\newcommand{\e}{{\bf e}}
\newcommand{\Id}{\mbox{Id}}

\newcommand{\be}[1]{\begin{equation}\label{#1}}
\newcommand{\ee}{\end{equation}}
\renewcommand{\overline}[1]{(#1)^{\rm c}}

\renewcommand{\d}{\mathrm{d}}
\renewcommand{\r}[1]{(\ref{#1})}
\renewcommand{\i}{\mathrm{i}}

\title[Support theorem for the Light-Ray transform  on Minkowski spaces]{Support theorem for the Light-Ray transform of vector fields on Minkowski spaces}

\author[Siamak RabieniaHaratbar]{}



 \email{srabieni@purdue.edu}

\thanks{Partly supported by NSF Grants DMS~1301646 and DMS~1600327}

\begin{document}
\maketitle

\centerline{\scshape Siamak RabieniaHaratbar}
\medskip
{\footnotesize
 \centerline{Department of Mathematics, Purdue University}
 \centerline{West Lafayette, IN 47907, USA}
 }

 \bigskip

 \centerline{(Communicated by Allan Greenleaf)}

\begin{abstract}
   We study the Light-Ray transform of integrating vector fields on the Minkowski time-space $\R^{1+n}$, $n\ge 2$, with the Minkowski metric. We prove a support theorem for vector fields vanishing on an open set of light-like lines. We provide examples to illustrate the application of our results to the inverse problem for the hyperbolic Dirichlet-to-Neumann map.
\end{abstract}

\section{Introduction}

	Let $(M,g)$ be a Lorentzian manifold of dimension $1+n$, $n\ge 2$, with a Lorentzian metric $g$ of signature $(-,+,...,+)$. Given a weight $\kappa \in C^\infty (M \times S^{n-1})$,
	in general, the weighted Light-Ray transform of a vector field $f$ is defined by:
\be{1}
	L_\kappa f(\gamma)= \int \kappa(\gamma(s),\gamma^{\prime}(s)) f_i(\gamma(s)) {\gamma^\prime}^i(s)  \ ds,
\ee
	where $\gamma=\gamma_{x,\theta}$ is the family of future pointing light-like geodesics (null-geodesics) on $M$ in the direction of $(1,\theta)$, $\theta \in S^{n-1}$. We choose a certain parametrization for the family of light-like geodesics $\gamma$ and require the weight function $\kappa$ to be positively homogeneous of degree zero in its second variable. The homogeneity of $\kappa$ makes the parameterization of light-like geodesics independent.

	Light-ray transform has been attracting a growing interest recently, due to its wide range of applications. One major application of this transform is in the study of hyperbolic equations with time-dependent coefficients to recover the lower order terms from boundary or scattering information, see, e.g., \cite{40, 32, 31, 48} \cite{
	 49, 33, 4, 15} and also recovering the lower order terms of time-independent hyperbolic equations \cite{3,26}.

	In the case where $f$ is a function and supported in the cylinder $\R\times B(0,R)$ with tempered growth in the Minkowski space, Stefanov in \cite{39} has shown that $Lf$ determines $f$ uniquely. The fact that $Lf$ recovers the Fourier transform $\hat{f}$ of $f$ (w.r.t. all variables) in the space-like cone $|\tau| < |\xi|$ in a stable way, is used to show that the potential in the wave equation is uniquely determined by the scattering data. Moreover, since $\hat{f}(\tau,\xi)$ is analytic in the $\xi$ variable (with values distributions in the $\tau$ variable), then one can fill in the missing cone by analytic continuation in the $\xi$ variable. In a recent work by Stefanov \cite{39}, analytic microlocal methods are applied to show support theorems and injectivity of $L_\kappa$ for analytic metrics and weights (on an analytic manifolds $M$) for functions. In particular, the results in \cite{40} are generalized to a local data and independent of the tempered growth for large $t$.

	Analytic microlocal method were already used in many works. Boman and Quinto in \cite{6, 7} proved support theorems for Radon transforms with flat geometry and analytic weights, see also \cite{30}. For related results using same techniques, we refer to \cite{22, 23} where the support theorem is proved on simple analytic manifolds, see also \cite{11}. Uhlmann and Vasy in \cite{46} used the scattering calculus to prove a support theorem in the Riemannian case near a strictly convex point of a hypersurface in dimensions $n \ge 3$ without the analyticity condition, see also \cite{42, 43, 44}.

	Analytic and non-analytic microlocal analysis have been used to prove the injectivity and stability estimates for tensor fields of order two and higher. For the tensor fields of order two, Pestov and Uhlmann in \cite{28} proved the unique recovery of Riemannian metric on a compact and simple Riemannian surface with boundary, see also Sharafautdinov \cite{37} for related results. More general results on injectivity up to potential fields of the geodesic ray transform for tensor fields of any order on Riemannian manifold can be found in \cite{27}. In \cite{43, 44}, a generic s-injectivity up to potential fields and a stability estimate are established on a compact Riemannian manifold $M$ with non-necessarily convex boundary and with possible conjugate points. Microlocal method for tomographic problems is used to detect singularities of the Lorentzian metric of the Universe using measurements of the Cosmic Microwave Background (CMB) radiation in \cite{24}. In \cite{25}, it is described that which singularities are visible and which cannot be detected. In \cite{23}, a Helgason's type of support theorem is proved for the geodesic ray transform of symmetric 2-tensor fields on a Riemannian manifold (with boundary) with a real-analytic metric $g$. It is shown that the tensor field can be recovered up to a potential field. In \cite{18}, authors studied the problem of recovery of a covector field on simple Riemannian manifold with weight. Under some condition on weight, the recovery up to potential field and uniqueness are shown. See also \cite{10}, for the inversion of three dimensional $X$-ray transform of symmetric tensor fields of any order with sources on a curve.

	Our main goal in this paper is to study the local and analytic microlocal invertibility of the operator $L$ acting on vector fields on the Minkowski time-space $\R^{1+n}$ when the weight $\kappa$ is simply one. We study the following operator $L$:
\be{2}
	Lf(x,\theta)= \int f(s,x+s\theta)\cdot (1,\theta) \ ds,
\ee
	where $\gamma=\gamma_{x,\theta}=(s,x+s\theta)$ is the family of future pointing light-like lines (light-rays) on the Minkowski time-space $\R^{1+n}$ in the direction of $(1,\theta)$, with $|\theta|=1$. Note that above operator is a special case of the operator defined in \r{1}.

	The main novelty of our work is that $Lf$ is known only over an open set of light-like lines, $\Gamma$, on the Minkowski time-space $\R^{1+n}$ (the Incomplete data case). Our results can be considered as a Helgason's type support theorem. The global invertibility (injectivity) of the operator $L$ (the Complete data case) up to potential fields is already established, see for example \cite{10, 50}.

	We generalize the method used in \cite{39} to study the stable recovery of the analytic wave front set of vector field $f$ instead of functions, and prove a support theorem in the Minkowski time-space $\R^{1+n}$. To prove our results, we apply the analytic stationary phase approach by Sj\"ostrand \cite{38} already used by the Stefanov \cite{39} and Uhlmann in \cite{44}, see also \cite{11, 22, 23}.

	This paper is organized as follow: Section one is an introduction. In section two we state some definitions and our main result. Fourier analysis of the light-ray transform $L$ in the Minkowski time-space $\R^{1+n}$ is studied in section three. In section four, we first introduce the notion of Analytic Wave Front Set of a vector-valued distribution $f$. We then consider the partial data and apply analytic microlocal analysis argument and stationary phase method to recover the analytic wave front set of vector field $f$. In section five, we state an analytic continuation result and we prove our main theorem. Our examples in the last section illustrate how our results imply the inverse recovery of a smooth potential field for the hyperbolic Dirichlet-to-Neumann map.

\section{Main result}
	We first state some definitions and a proposition which are necessary for our main result.

\begin{definition}
	We call a vector $u = (u_0 ,u^\prime)$ space-like if $|u_0| < |u^\prime|$. Vectors with $|u_0| > |u^\prime|$ are called time-like. Light-like vectors are those for which we have $|u_0| = |u^\prime|$.
\end{definition}

\begin{definition}
	We say the set $K$ is light-like convex if for any two pints in $K$, the light-like geodesic connecting them lies in $K$.
\end{definition}

\begin{definition}
	Let $K$ be a subset of Minkowski time-space $\R^{1+n}$. We say $K$ expands with a speed less than one if
	\begin{align*}
		K \subset \lbrace (t,x):|x|\leqslant C|t|+R \rbrace, \quad \text{ for some \quad $0<C<1$ , $R > 0$}.
	\end{align*}
\end{definition}

\begin{remark}
	\normalfont
	Definition 2.3 allows us to integrate over a compact interval. In fact, if the $\supp f$ in such a set expands with speed less than one, then the operator defined by \r{2} is integrating over a compact set including $(x_0,\theta_0) \in \R^n \times S^{n-1}$. In other words, the integral of $f$ and $\chi f$ have the same light-ray transform near $(x_0,\theta_0)$, where  the function $\chi$ is a smooth cut-off with property $\chi=1$ in a neighborhood of $(x_0,\theta_0)$.
\end{remark}

	From now on, we study the operator defined by $\r{2}$. We know that, any three-dimensional vector field $f=(f_0,f_1,f_2)$, has a three-dimensional $\curl f$. In other words, one may work with the $\curl f$ to do the analytic recovery of the analytic wave front set. This, however, is not the case for any vector field $f$ with dimension $n>3$ as the generalized $\curl f$, $\d f$, does not have the same dimension as the vector field $f$ does. This motivates us to introduce an appropriate operator where it forms an $n$-dimensional parametrized vector field with all the necessary components of $\d f$ for the analytic recovery process. We now state our first proposition.

\begin{proposition}
	Let $f=(f_0,f_1,\dots,f_n) \in \mathcal{C}^1(\R^{1+n},\mathbb{C}^{1+n})$ be such that $|f|$ and $|\partial f_i /\partial x^i|$ are bounded by $C(1+|x|)^{-1-\epsilon}$ with some $\epsilon>0$ and constant $C>0$. Then for any $(x,\theta)\in (\R^n \times S^{n-1})$ and $v \in \R^{n},$
\be{3}
	(v.\nabla_x)Lf(x,\theta)=\int_{\gamma_{x,\theta}} \tilde{f}_v(x) \cdot (1, \theta) \ ds,
\ee
	where $\tilde{f}_v(x)=(\tilde{f}_{0_v},\tilde{f}_{1_v},\dots, \tilde{f}_{n_v})(x)\in\mathcal{S}(\R^{1+n})$ is a parametrized vector field with
$$
	\tilde{f}_{i_v} (x)=\sum_{0\leq j \leq n} f_{ij} (x)(0,v)^j _{(i)}=\sum_{0\leq j \leq n} (\partial_jf_i-\partial_if_j)(x)(0,v)^j _{(i)}, \quad i=1,2,\dots, n.
$$
	Here by $(0,v)_{(i)},$ we mean the $i$-th component of $(0,v)$ is excluded.
\end{proposition}

\begin{proof}
	We show this result for $n=3$. The proof for higher dimension is analogous. Let $f\in \mathcal{C}^1(\R^{1+3},\mathbb{C}^{1+3})$ and fix $(1,\theta) \in \R \times S^{n-1}$. For $v \in \R^3$, we take the directional derivative of the operator $Lf$. Therefore,
\begin{align*}
	(v.\nabla_x)Lf(x,\theta)=\int_R (v.\nabla_x)f(s,x+s\theta).(1,\theta)ds.
\end{align*}
	On the other hand, by the Fundamental Theorem of Calculus,
\begin{align*}
	\int_R \frac{d}{ds}[f(s,x+s\theta).(0,v)] \ ds =0.
\end{align*}
	Subtracting above identities, we have
\begin{align*}
	(v.\nabla_x)Lf(x,\theta)=\int_R ((0,v).\nabla_z)f(s,x+s\theta).(1,\theta)- \frac{d}{ds}[f(s,x+s\theta).(0,v)]\ ds.
\end{align*}
	Note here that we used $z=(t,x)\in \R^{1+n}$ to balance the dimension of the two terms on the right hand side of above equation. Expanding the right hand side and rearranging all terms with respect to components of $(1,\theta),$ i.e., $1, \theta^1,\theta^2,\theta^3$, we get
$$
	\int_{\gamma_{x,\theta}} [v^1 (\partial_1f_0 - \partial_0f_1) + v^2 (\partial_2f_0 - \partial_0f_2) + v^3 (\partial_3f_0 - \partial_0f_3)]
$$
$$
	+[v^2 (\partial_2f_1 - \partial_1f_2) + v^3 (\partial_3f_1 - \partial_1f_3)] \theta^1
	+[v^1 (\partial_1f_2 - \partial_2f_1) + v^3 (\partial_3f_2 - \partial_2f_3)] \theta^2
$$
\vspace{.1mm}
$$
	+[v^1 (\partial_1f_3 - \partial_3f_1) + v^2 (\partial_2f_3 - \partial_3f_2)] \theta^3  \ ds.
$$
	Therefore,
$$
	(v.\nabla_x)Lf(x,\theta)=\int_{\gamma_{x,\theta}} \tilde{f}_v(x) \cdot (1, \theta) \ ds.
$$
	Here $\gamma_{x,\theta}$ is the light-like lines parameterized by their points of intersection with $t=0$ and direction $(1,\theta)$.
\end{proof}

\begin{remark}
	\normalfont
	i) For $n=2$, setting $v$ to be $(1,0)$ and $(0,1)$ yields to the following identities.
$$
	\partial_1 Lf(x,\theta)= \int_{\gamma_{x,\theta}} (\partial_1 {f_0}-\partial_t {f_1})+\theta ^2 (\partial_1 {f_2}-\partial_2 {f_1}) \ ds= \int_{\gamma_{x,\theta}} (-c_2 +\theta^2 c_0) \ ds.
$$
$$
	\partial_2 Lf(x,\theta)= \int_{\gamma_{x,\theta}} (\partial_2 {f_0}-\partial_t {f_2})+\theta ^1 (\partial_2 {f_1}-\partial_1 {f_2}) \ ds= \int_{\gamma_{x,\theta}} (c_1 -\theta^1 c_0) \ ds.
$$
	where $(c_0,c_1,c_2)=:\curl f$. Similar results can be seen in [\cite{41}, Proposition~2.8].

	ii) The vector field $\tilde{f}_v$ has the following property: for any $v \in \R^n$,
$$
	(0,v)\cdot \tilde{f}_{v}(x) = v^1\tilde{f}_{1_v}(x)+ v^2\tilde{f}_{2_v}(x)+ \dots + v^n\tilde{f}_{n_v}(x)=0.
$$
	This is analogous to solenoidal condition for vector fields in the Fourier domain.
 
 iii) Each component of $\tilde{f}_v$ is a superposition of components of $\curl f$ (for $n=2$) and of the generalized $\curl$, $\d f$ (for $n\ge 3$.) This is a very important property since it forms an overdetermined system of equations which helps us to recover the $\curl$ and generalized $\curl$, $\d f$.
 
  iv) Clearly
$$
	Lf=0 \ \Longrightarrow \ (v.\nabla_x)Lf=0, \quad \forall v \in \R^n.
$$
	For the case where the vector field $f$ is compactly supported,
$$
	Lf=0 \ \equiv \ (v.\nabla_x)Lf=0, \quad \forall v \in \R^n.
$$
	In fact, the directional derivative of $Lf$ with respect to $x$ is zero for all $v \in \R^n$, which implies that $Lf$ is constant. Now $f$ is compactly supported, therefore $Lf=0$. 
\end{remark}
	Above properties motivate us to work with $\tilde {f}_v$ and $(v.\nabla_x)L$ instead of $f$ and $Lf$ in the following sections. Our main result is a support theorem in the spirit of Theorem 2.1 \cite{39} as follow:

\begin{theorem}
	Let $n\ge 2$ and $f \in \mathcal{E^\prime} (\R^{1+n})$ be so that $\supp f$ expands with a speed less than one. Let $G$ be an open and connected neighborhood of $(x_0,\theta_0) \in \R^n \times S^{n-1}$ and $\gamma_{x_0,\theta_0}$ be a light-like line with direction $\theta_0$ passing through the point $x_0$. 

	i) For $n=2$, if $Lf(x,\theta) = 0$ in $G_{\pm}$ and if $\gamma_{x_0,\theta_0}$ does not intersect $\supp \curl f$, then none of the light-like lines $\gamma_{x,\theta}$, $(x,\theta) \in G_{\pm}$, does. Here $G_{\pm}$ is an open and connected neighborhood of $(x_0,\pm\theta_0)$ in $\R^2 \times S^{1}$. 

	ii) For $n \ge 3$, if $Lf(x,\theta) = 0$ in $G$ and if $\gamma_{x_0,\theta_0}$ does not intersect $\supp \d f,$ then none of the light-like lines $\gamma_{x,\theta}$, $(x,\theta) \in G$, does.
\end{theorem}

\begin{remark}
	\normalfont
	For $n=2,$ we require the operator $L$ to be known in two different directions $(1,\theta)$ and $(1,-\theta)$. For discussion, we refer the reader to Fourier analysis in the following section. It is shown that, the ellipticity is lost when the operator $L$ is known only for one direction $(1,\theta)$.
\end{remark}

\section{Fourier analysis in the Minkowski case}
	In this section, we consider the case where the light-ray transform is known over all light-like lines (complete data). This allows us to do the Fourier analysis by fixing the initial point and stay in a small neighborhood of the direction $(1,\theta)$. As we mentioned above, this case has been already studied. We do this analysis to have some insight for the microlocal analysis part of our study.

	The following proposition is some preliminary results which stats that in the space-like cone $\lbrace(\tau,\xi): |\tau| < |\xi|\rbrace$, the operator $(v.\nabla_x)L$ recovers the Fourier transform of the $\curl f$ and the generalized $\curl$, $\d f,$ for $n=2$ and $n\ge 3$, respectively.

\begin{proposition}
	Let $f\in\mathcal{S}(\R^{1+n})$.

	i) For $n=2$, if $Lf(x,\theta)=0$ for all $x$ and for $\theta$ near $\pm\theta_0$, then $\mathcal{F}(\curl f)=0$ for $\zeta$ close $\zeta_0$, where $\zeta_0$ is the unique space-like vector up to re-scaling with the property $(1,\pm\theta_0)\cdot \zeta_0=0$.

	ii) For $n\geq 3$, if $Lf(x,\theta)=0$ for all $x$ and for $\theta$ near $\theta_0$, then $\mathcal{F}(\d f)=0$ for all $\zeta$ near the set $\lbrace \zeta \big| (1,\theta)\cdot \zeta=0 \rbrace$.
\end{proposition}
\begin{proof}
	i) Let $\zeta^0=(\tau^0,(\xi^1)^0,(\xi^2)^0)$ be a fixed space-like vector, and without loss of generality assume that $\theta_0=\pm \e_2 \in \R^2$ such that $(1,\pm\theta_0)\cdot \zeta^0=0$. One has
\begin{align*}
	\tau^0 \pm (\xi^2)^0=0 \quad \text{which implies that $\tau^0=(\xi^2) ^0=0$}.
\end{align*}
	Therefore, the vector $\zeta^0$ has to be in the form of $(0,(\xi^1)^0,0),$ which means, up to re-scaling it is a unique $\zeta^0$ with property $(1,\pm\theta_0)\cdot \zeta^0=0$. Hence, one may choose $\zeta^0=\e^1 \in \R^{1+2}$. Note that this choice of $\zeta^0$ can be done since we may apply Lorentzian transformation to any fixed space-like vector and transform it to $\e^1$.
	We first state the \textbf{Vectorial Fourier Slice Theorem} for a general set of lines:
$$
	\hat{f}(\zeta)\cdot \omega = \int_{\omega^\bot} e^{-\i z\cdot \zeta} Lf(z,\omega)\ \d S_z, \quad  \quad \text{$\forall \omega\bot \zeta,$ \ $\forall f\in L^1(\R^n)$.}
$$
	To prove this, note that the integral on the RHS equals
$$
	\int_{\omega^\bot} \int_\R e^{-\i z\cdot \zeta} f_i(z+s\omega) \omega^i \d s \d S_z=\omega^i\int_{\omega^\bot} \int_\R e^{-\i z\cdot \zeta} f_i(z+s\omega) \d s \d S_z.
$$
	Set $x = z +s\omega$. Then, it is easy to see that when $\omega \bot \zeta$, we have $x\cdot\zeta = z \cdot \zeta$ and therefore above integral equals $\hat{f}_i(\zeta)$. In this paper, we apply the Vectorial Fourier Slice theorem when the set of lines is restricted to a set of light-like lines.

	By assumption $(v.\nabla_x)Lf=0$ for all $v \in \R^2$. Let $v$ be an arbitrary but fixed vector in $\R^2$. For $\zeta=\e^1$, by Vectorial Fourier Slice theorem we have
\be{4}
	0=\hat{\tilde{f}}_{v}(\zeta)\cdot (1,\theta) = \hat{\tilde{f}}_{0_v}(\zeta) + \hat{\tilde{f}}_{1_v}(\zeta)\theta^1 + \hat{\tilde{f}}_{2_v}(\zeta)\theta^2, \quad \forall (1,\theta)\bot \zeta.
\ee
	Since $(1,\pm \theta_0)\cdot \zeta=0$, above equation implies
$$
	\left\{\begin{array}{ll}
	\hat{\tilde{f}}_{v_0}(\zeta)=v^1(\xi^1\hat{f}_0-\tau \hat{f}_1) + v^2(\xi^2\hat{f}_0-\tau \hat{f}_2)=0,\\
	\hat{\tilde{f}}_{v_2}(\zeta)=v^1(\xi^1\hat{f}_2-\xi^2 \hat{f}_1)=0.
	\end{array}\right.
$$
	The vector $v$ is arbitrary, therefore one may choose two linearly independent vectors, say $v_1=\theta^\bot$ and $v_2=\theta$, to conclude
$$
	\xi^1\hat{f}_0-\tau \hat{f}_1 = \xi^2\hat{f}_0-\tau \hat{f}_2 = \xi^1\hat{f}_2-\xi^2 \hat{f}_1 =0 \ \Longrightarrow \ \mathcal{F}(\curl f)(\zeta)=0.
$$

	Now let $\zeta=(\tau, \xi)\in \R^{1+2}$ be any non-zero space-like vector. We solve the equation $(1,\theta)\cdot\zeta=0,$ for $\theta$. Set $\theta=a\xi+ b\xi^\bot$. Therefore,
\begin{align*}
	-\tau=\theta\cdot\xi=(a\xi+ b\xi^\bot)\cdot\xi=a |\xi|^2 \quad \Longrightarrow \quad a= \frac{-\tau}{|\xi|^2}.
\end{align*}
	On the other hand,
\begin{align*}
	1=|\theta|^2=(a^2+b^2)|\xi|^2 \quad \Longrightarrow \quad b=\pm \frac{1}{|\xi|^2}\sqrt{-\tau^2+|\xi|^2}.
\end{align*}
	For $\xi\in \R^2$, we set
\be{5}
	\theta=\theta_\pm(\zeta)=\frac{1}{|\xi|^2}(-\tau\xi^1\mp \sqrt{-\tau^2+|\xi|^2}\xi^2, -\tau\xi^2\pm \sqrt{-\tau^2+|\xi|^2}\xi^1).
\ee
	Clearly $(1,\theta_\pm(\zeta))\cdot \zeta=0$ and $\theta_\pm$ are the only two choices with the property $|\theta_\pm(\zeta)|=1$. In order to have $\zeta$ close to $\zeta_0=\e^1$, we require $\theta_\pm ^1=0$, so we have
\begin{align*}
	-\tau\xi^1\mp \sqrt{-\tau^2+|\xi|^2}\xi^2=0.
\end{align*}
	Since $\zeta$ is a space-like vector, $\sqrt{-\tau^2+|\xi|^2}$ is non-zero which implies that $\tau \xi^1=\xi^2=0$. Note that $\xi^1$ is not zero, otherwise $\theta_{\pm}(\zeta)$ would be undefined. This forces $\tau$ to be zero, and therefore $\zeta \approx\zeta_0=\e^1$. In particular, this implies that $\theta=\theta_\pm(\zeta)$ is analytic near $\zeta^0=\e^1 \in \R^{1+2}$ with $\theta_\pm(\zeta^0)=\theta_0=\pm \e_2 \in \R^2$. Hence, $\theta_\pm (\zeta)$ is within a neighborhood of $\pm \theta_0$, $\theta \approx \pm \theta_0$, if $\zeta$ is within a neighborhood of $\zeta_0$, $\zeta \approx \zeta_0$.
	Considering our choices of direction $\theta_\pm(\zeta)$, the equation \r{4} can be written as
\be{6}
	\hat{\tilde{f}}_{0_v}(\zeta)+\frac{1}{|\xi|^2}(-\tau\xi^1 + \sqrt{-\tau^2+|\xi|^2}\xi^2) \hat{\tilde{f}}_{1_v}(\zeta)+\frac{1}{|\xi|^2}(-\tau\xi^2 - \sqrt{-\tau^2+|\xi|^2}\xi^1) \hat{\tilde{f}}_{2_v}(\zeta)=0,
\ee
	and
\be{7}
	\hat{\tilde{f}}_{0_v}(\zeta)+\frac{1}{|\xi|^2}(-\tau\xi^1 - \sqrt{-\tau^2+|\xi|^2}\xi^2) \hat{\tilde{f}}_{1_v}(\zeta)+\frac{1}{|\xi|^2}(-\tau\xi^2 + \sqrt{-\tau^2+|\xi|^2}\xi^1) \hat{\tilde{f}}_{2_v}(\zeta)=0.
\ee
	Subtract $\r{7}$ from $\r{6}$ to get
\be{8}
	0=\xi^2 \hat{\tilde{f}}_{1_v}(\zeta)-\xi^1 \hat{\tilde{f}}_{2_v}(\zeta)=(\xi \cdot v)[ \xi^2 \hat{f}_1(\zeta) - \xi^1 \hat{f}_2(\zeta)], \quad \text{for the fixed $v \in \R^2$}.
\ee
	Multiplying $\r{6}$ by $\xi^1$ and using $\r{8}$, we get
\be{9}
	0=\tau \hat{\tilde{f}}_{1_v}(\zeta)-\xi^1 \hat{\tilde{f}}_{0_v}(\zeta)=(\xi \cdot v)[\tau  \hat{f}_1(\zeta) - \xi^1  \hat{f}_0(\zeta)], \quad \text{for the fixed $v \in \R^2$}.
\ee
	Note that for $i=0,1,2,$ we expanded $\hat{\tilde{f}}_{i_v}(\zeta,v)$ in \r{8} and \r{9}, and rearranged both equations in terms of $\hat{f}_i(\zeta)$, to get the RHS of above equations. Now set $v=\xi$, therefore $\xi \cdot v=|\xi|^2 \not =0$ and
\be{10}
	\xi^2 \hat{f}_1(\zeta) - \xi^1 \hat{f}_2(\zeta)=\tau  \hat{f}_1(\zeta) - \xi^1  \hat{f}_0(\zeta)=0.
\ee
	Clearly $\xi^2 \hat{f}_1(\zeta) - \xi^1 \hat{f}_2(\zeta)=0$ implies that $\hat{\tilde{f}}_{1_v}(\zeta)=\hat{\tilde{f}}_{2_v}(\zeta)=0$ for $v \in \R^2$. Plugging $\hat{\tilde{f}}_{1_v}(\zeta)=0$ into the LHS of \r{9}, we conclude  $\xi^1\hat{\tilde{f}}_{0_v}(\zeta)=0$. The vector $\zeta$ is space-like with property $\xi^1 \not =0$, therefore
$$
	\hat{\tilde{f}}_{0_v}(\zeta)=v^1(\xi^1\hat{f}_0-\tau \hat{f}_1) + v^2(\xi^2\hat{f}_0-\tau \hat{f}_2)=0.
$$
	Since $v$ is arbitrary in $\R^2$, any two linearly independent vectors $v_1$ and $v_2$ implies that $\xi^2\hat{f}_0-\tau \hat{f}_2=0$. Notice that one may use the equation on the RHS of \r{10} and the fact that $v$ is an arbitrary vector to have the same conclusion. This shows that all three components of $\curl f$ in Fourier domain are zeros, and thus $\mathcal{F}(\curl f)(\zeta)=0$.

	ii) Let first $n=3$ and $\zeta=\zeta^0$ be a fixed non-zero space-like vector. Applying the Lorentzian transformation, we may assume that $\zeta^0=\e^{2}:=(0,0,1,0)\in \R^{1+3}$. Set
$$
	\theta(a)=\sin (a) \e_1 +\cos (a) \e_3 =(\sin (a),0,\cos (a)).
$$
	Clearly $\big|\theta(a)|=1$, $\theta_0=\theta(0)=\e_3$, and $(1,\theta(a))\cdot \zeta=0$.

	By assumption $(v.\nabla_x)Lf=0$ for all $v \in \R^3$. Let $v$ be an arbitrary fixed vector in $\R^3$. For $\zeta=\e^2$, by Vectorial Fourier Slice theorem
$$
	0=\hat{\tilde{f}}_{v}(\zeta)\cdot (1,\theta) = \hat{\tilde{f}}_{0_v}(\zeta) + \hat{\tilde{f}}_{1_v}(\zeta)\theta^1 + \hat{\tilde{f}}_{2_v}(\zeta)\theta^2 + \hat{\tilde{f}}_{3_v}(\zeta)\theta^3, \quad \forall (1,\theta)\bot \zeta
$$
	Plugging $\theta=\theta(a)$ into above equation we get
\be{11}
	0=\hat{\tilde{f}}_{v}(\zeta)\cdot (1,\theta) = \hat{\tilde{f}}_{0_v}(\zeta) + \hat{\tilde{f}}_{1_v}(\zeta)\sin(a) + \hat{\tilde{f}}_{3_v}(\zeta)\cos(a).
\ee
	Differentiating above equation with respect to parameter $a$ once and twice, we get
$$
	\left\{\begin{array}{ll}
	\ \ \hat{\tilde{f}}_{1_v}(\zeta)\cos(a) - \hat{\tilde{f}}_{3_v}(\zeta)\sin(a)=0,\\
	-\hat{\tilde{f}}_{1_v}(\zeta)\sin(a) - \hat{\tilde{f}}_{3_v}(\zeta)\cos(a)=0. \\
	\end{array}\right.
$$
	It is easy to see that the last two equations imply that $\hat{\tilde{f}}_{1_v}(\zeta)=\hat{\tilde{f}}_{3_v}(\zeta)=0$ for $v \in \R^2$. Now by equation \r{11} we conclude that $\hat{\tilde{f}}_{0_v}(\zeta)=0$ for $v \in \R^2$.

	Our goal is to show the Fourier transform of the generalized $\curl$ of $f$, $\mathcal{F}(\d f),$ is zero. Let $v=(-\cos(a),1,\sin(a))\in\theta^\bot$ and plug it into $\hat{\tilde{f}}_{i_v}(\zeta)=0$ for $i=0,1,3$. We have
$$
	\left\{\begin{array}{ll}
	-(\xi^1\hat{f}_0-\tau\hat{f}_1)\cos (a) + (\xi^2\hat{f}_0-\tau\hat{f}_2) + (\xi^3\hat{f}_0-\tau\hat{f}_3)\sin(a)=0,\\
	\ \ (\xi^2\hat{f}_1-\xi^1\hat{f}_2) + (\xi^3\hat{f}_1-\xi^1\hat{f}_3)\sin(a)=0, \\
	-(\xi^1\hat{f}_3-\xi^3\hat{f}_1)\cos(a) + (\xi^2\hat{f}_3-\xi^3\hat{f}_2)=0.\\
	\end{array}\right.
$$
	One may repeat above differentiation argument for the first equation to conclude that
$$
	\xi^1\hat{f}_0-\tau\hat{f}_1=\xi^2\hat{f}_0-\xi^0\hat{f}_3=\xi^3\hat{f}_0-\xi^0\hat{f}_3= 0.
$$
	Using the same argument simultaneously for the second and third equations implies that
$$
\xi^2\hat{f}_1-\xi^1\hat{f}_2= \xi^2\hat{f}_3-\xi^3\hat{f}_2= \xi^3\hat{f}_1-\xi^1\hat{f}_3 =0.
$$

	Therefore, $\mathcal{F}(\d f)(\zeta)=0$ for a fixed $\zeta=\e^2$. Note that, one may choose three linearly independent vectors $v_1,v_2, v_3\in \R^3$, and conclude the same result.

	To have the result for an arbitrary $\zeta$, we use the fact that the Lorentzian transformation is transitive and rotates every space-like vector to a space-like vector. Let $\mathcal{L}_{\zeta^0}$ be a Lorentzian transformation with the property $\mathcal{{L}}^{-1} _{\zeta^0} \zeta=\zeta^0=\e^2$ and let $\mathcal{L}$ with $\mathcal{L}x=y$ be a Lorentzian trasformation whose representation in Fourier domain is given by $\mathcal{L}_{\zeta^0}$.
	By the definition of Fourier transform, one has
$$
	\mathcal{F}[(\d f){\mathcal{L}(.)}](\zeta)=\int (\d f)(\mathcal{L}x) e^{ix\cdot \zeta} dx=\int (\d f)(y) e^{i\mathcal{L}^{-1}y\cdot \zeta} \big| det\mathcal{L}^{-1}\big| dy
$$
$$	
	=\int (\d f)(y) e^{iy\cdot \mathcal{L}_{\zeta^0}^{-T}\zeta} \big| det\mathcal{L}^{-1}\big| dy.
$$
	Therefore,
$$
	\mathcal{F}[(\d f){\mathcal{L}(.)}](\zeta)=\big| \det \mathcal{L}^{-1}\big| \mathcal{F}(\d f){(\mathcal{L}^{-T}\zeta)}.
$$
	But $\mathcal{L}^{T}=\mathcal{L}$ and $\mathcal{L}^{-1}(\zeta)=\zeta^0$, hence
$$
	\mathcal{F}[(\d f){\mathcal{L}(.)}](\zeta)=\big| \det \mathcal{L}^{-1}\big| \mathcal{F}(\d f)(\zeta^0)=0,
$$
	since $\mathcal{F}(\d f)(\zeta^0)=0$. This proves that for any space-like vector $\zeta$ near $\zeta^0$, the Fourier transform of generalized $\curl$ of $f$ vanishes as desired.
	For the general case $n > 3 $, one needs to choose $n$ linearly independent vectors $v_1,v_2, \dots, v_n \in \R^n$ to show $\mathcal{F}(\d f)(\zeta)=0$.
\end{proof}

\begin{corollary}
	Let $f\in\mathcal{C}_0 ^\infty (\R^{1+n})$.

	i) For $n=2$, if $Lf(x,\theta)=0$ \ for $\theta$ near $\pm\theta_0$, then $f$ is a smooth potential field with compact support, that is, $f=d\phi$ with some $\phi(x) \rightarrow 0$, as $|x|  \rightarrow \infty$.

	ii) For $n\geqslant 3$, \ if $Lf(x,\theta)=0$ \ for $\theta$ near $\theta_0$, then $f$ is a smooth potential field with compact support, that is, $f=d\phi$ with some $\phi(x) \rightarrow 0$, as $|x|  \rightarrow \infty$.
\end{corollary}

\begin{proof}
	i) The first part of Proposition 3.1 implies that $\mathcal{F}(\curl f)=0$. Since $f\in\mathcal{C}_0 ^\infty(\R^{1+2})$, we extend $f$ as zero outside of the $\supp f$. Now by analyticity of Fourier transform, $\mathcal{F}(\curl f)$ is zero everywhere. Applying the inverse Fourier transform implies that $\curl f=0$ everywhere. Since the time-space $\R^{1+2}$ is a simply connected domain, there exists a finite smooth function $\phi$ such that $f=d\phi$; in fact let $x \in \supp f$ and $x_0$ be a point outside of $\supp f$. Let $c(t)$ be a path connecting $x_0$ to $x$. We define $\phi$ as follow:
\begin{align*}
	\phi(x)= \int_{x_0} ^{x} f(c(t)) \cdot c^\prime(t) dt + \phi(x_0),
\end{align*}
	which is smooth and satisfies $f=\d \phi.$

	ii) By the second part of Proposition 3.1 we know that $\mathcal{F}(\d f)=0$. Similar argument as part $(i)$ shows $\d f=0$ and therefore, $f$ is a smooth potential vector field with compact support.
\end{proof}

\begin{remark}
	\normalfont
	i) For $n=2$ there are two discrete choices of directions, $\pm\theta_0$, and this is necessary to have the result. Following example shows that one cannot decrease the number of directions from two to one. Let $\zeta$ be a space-like vector and $\phi \in \mathcal{S}$ be supported in the interior of open cone $\lbrace |\tau| < |\xi| \rbrace$. Consider $\theta_\pm(\zeta)$ defined by \r{8} and set
$$
	\eta=\hat{f}(\zeta)=(1,\frac{\xi^2}{\sqrt{-\tau^2+|\xi|^2}},\frac{-\xi^1}{\sqrt{-\tau^2+|\xi|^2}})\hat{\phi}(\zeta).
$$
	Clearly $\eta$ is a non-zero vector field in the Schwartz space $\mathcal{S}$ and is in the kernel of light-ray transform as it solves $(1,\theta_+(\zeta))\cdot \hat{f}(\zeta)=0$. Notice that, $(1,\theta_-(\zeta))\cdot \hat{f}(\zeta)\not=0$. This example does not provide a compactly supported vector field, however, it shows this is an obstruction to consider only one light-ray and stay close to it for the reconstruction. For the Minkowski spaces of signature $1+3$ or higher, however, this is not an obstruction. For instance when $n=3$, one may consider a two-parameter family of directions, $\theta(a,b)$, near a fixed $\theta_0$ and do the reconstruction process by perturbation.

	ii) In above proposition for $n\ge 3$, to show the uniqueness results we performed differentiation which is not problematic. However, for the stable inversion results, a differentiation may not preserve the stability. In other words, one may choose several discrete values of non-zero parameters near zero to create an invertible linear system to get stability estimate results. For discussion, we refer the reader to proof of theorem 4.1 for $n\ge3$.
\end{remark}

	In the next section, we state a theorem on the recovery of analytic space-like singularities in the Minkowski case which is a tool to prove our main result.
\section{Microlocal recovery of analytic wave front set}
	In this section, we mainly follow the analytic microlocal analysis argument to show that we can recover all space-like analytic singularities of $f$ conormal to the light-like lines along with integration of operator $(v.\nabla_x)L$. (See also [\cite{39}, Lemma~3.1])

	We first introduce the \textbf{Analytic Wave Front Set} (or “analytic singular spectrum) of a vector-valued distribution. For the case of a scalar-valued distribution, the definition can be found in  [\cite{38}, Sj\"ostrand]. We recall that, there are three existing definitions due to Bros-Iagolnitzer \cite{9}, H\"ormander \cite{19}, and Sato \cite{35}. Bony \cite{8} and Sj\"ostrand have shown the equivalence of all these definitions. For a vector-valued distribution $f=(f_0,f_1,f_2,\dots,f_n)\in \mathcal{D}^\prime (X, \mathbb{C}^{1+n})$, we define the analytic wave front set of $f$, $\WFA(f)$, as the union of $\WFA(f_i)$. Note that, for the vector-valued distribution $f$, the analytic wave front set $\WFA(f)$ does not specify in which component $f$ is singular. In our work, we follow the Sj\"ostrand's exposition.

\begin{theorem}
	Let $f\in\mathcal{E}^\prime(\R^{1+n})$ and let $\gamma_{x_0,\theta_0}$ be a fixed light-like line so that  $\gamma_{x,\theta}(s)\not\in \supp f$ for $|s|\ge 1/C$  with some $C$ for all $(x,\theta)$ near $(x_0,\theta_0)$.

	i) For $n=2$, if $Lf(x,\theta)=0$ for all $x, \theta$ near $(x_0,\pm\theta_0)$, then $\WFA(\curl f)$ contains no space-like vectors conormal to $\gamma_{x_0,\pm\theta_0}$.

	ii) For $n\geqslant 3$, if $Lf(x,\theta)=0$ for all $x, \theta$ near $(x_0,\theta_0)$, then $\WFA(\d f)$ contains no space-like vectors conormal to $\gamma_{x_0,\theta_0}$.
\end{theorem}

\begin{proof}
	i) Let first $f\in \mathcal{C}^1(\R^{1+2}).$ By assumption $Lf=0$ only near $\gamma_{x_0,\theta_0}$, therefore a localization is needed. We choose a local chart for the lines close to $\gamma_{x_0,\theta_0}$, and without loss of generality we may assume that $x_0=0$ and $\theta_0=\pm \e_2$. So we have $\gamma_0=\gamma_{0,\e_2}=(s,0,s)$.

	Since $Lf=0$, we have $(v.\nabla_x)Lf=0$ for $v \in \R^n.$  Let $v$ be an arbitrary fixed vector in $\R^n$ and $\zeta^0\not=0$ be a space-like vector conormal to $\gamma_0$ at $x_0=0$ with property $(1,\theta_0)\bot \zeta^0$. Applying the Lorentz transformation, we may assume that $\zeta^0=\e^{1}:=(0,1,0)\in \R^{1+2}$. Our goal is to show that $(0,\zeta^0)\not \in \WF_A(\curl f)$.

	Let $\chi_N\in C_0^\infty(\R^2)$ be supported in $B(0,\eps)$, with $\eps>0$ and $\chi_N=1$ near $x_0=0$ so that
\be{12}
	|\partial_{x}^\alpha\chi_N|\le (CN)^{|\alpha|}, \quad \text{for $|\alpha|\le N$}.
\ee
	Then for $0<\eps\ll1$, $\lambda>0$, and $\theta$ near $\theta_0$,
	\[
	0= \int e^{\i \lambda x\cdot\xi} (\chi_N (v.\nabla_x)Lf)(x,\theta)\,\d x=
	\iint e^{\i\lambda x\cdot\xi} \chi_N (x)  \tilde{f}_v(\gamma_{x,\theta}(s))\cdot(1,\theta) \ \d s \ \d x.
	\]
	If $(1,\theta)\cdot\zeta=0$ with $\zeta=(\tau,\xi)$, then $\gamma_{x,\theta}\cdot\zeta=(s,x+s\theta)\cdot\zeta=x\cdot\xi$. Performing a change of variable $z=\gamma_{x,\theta}$ in above integral yields to
\be{13}
	0= \int e^{\i \lambda x\cdot\xi} (\chi_N (v.\nabla_x)Lf)(x,\theta) \d x=
	\int e^{\i\lambda x(z,\theta)\cdot\xi} a_N (z,\theta) \tilde{f}_{i_v}(z)(1,\theta)^i \d z
\ee
$$	
	=\int e^{\i\lambda z\cdot\zeta} a_N (z,\theta) \tilde{f}_{i_v}(z)(1,\theta)^i \d z,
$$
	when $(1,\theta)\cdot\zeta=0$. Notice that $a_N(0,\theta)=1$. 

	Now let $\zeta$ be a space-like vector near $\zeta^0$ and set $\theta=\theta_\pm(\zeta)$ (see \r{5}). Plugging $\theta_\pm(\zeta)$ into \r{13}, we get
\be{14}
	\int e^{\i\lambda z\cdot\zeta} \tilde{a}_N (z,\zeta) \tilde{f}_{i_v}(z)(1,\theta_{\pm}(\zeta))^i\,\d z=0, \quad \text{near $\zeta=\e^{1}$}.
\ee
	Here $\tilde{a}_N (z,\zeta)=a_N (z,\theta)$ where $\tilde{a}_N (0,\zeta)=1$. Note also that for $\zeta\approx\zeta_0$, we have $\theta(\zeta) \approx \theta_0$.

	In the next step, we apply the complex stationary phase method [of Sj\"ostrand \cite{38}, similar to the case where it is applied to the Calder\'{o}n problem with partial data in \cite{20} and to the integral geometry problem in \cite{11, 44}.] We need to analyze the phase function and its critical points.

	Fix $0<\delta\ll1$ and let $\chi_\delta$ be the characteristic function of the unit ball $B(0,\delta)$ in $\mathbf{R}^{1+2}$. With some $w$, $\eta\in \mathbf{R}^{1+2}$ close to $w=0$, $\eta=\e^{1}$, multiply the LHS of \r{14} by
$$
	\chi_\delta(\zeta-\eta) e^{ \i \lambda(\i ( \zeta-\eta)^2/2 -  w\cdot\zeta) }
$$
	and integrate w.r.t.\ $\zeta$ to get
\be{15}
	\iint e^{\i\lambda\Phi(z,w,\zeta,\eta)} b_N(z,\zeta,\eta)\tilde{f}_{i_v}(z)(1,\theta_{\pm}(\zeta))^i\,\d z\, \d \zeta=0,
\ee
	where $b_N= \chi_\delta(\zeta-\eta)\tilde{a}_N$ is a new amplitude and
	\[
	\Phi = (z-w)\cdot\zeta +  \i (\zeta-\eta)^2/2.
	\]
	Consider the phase function $\zeta \rightarrow \Phi$. If $w=z$, there is a unique real critical point $\zeta_c =\eta$, with property $\Im \Phi_{\zeta \zeta}> 0$ at $\zeta=\zeta_c$. For $w\not=z$, the phase $\Phi$, as function of $\zeta$, has a unique critical point $\zeta_c =\eta +\i (z-w). $
	We now split the $z$ integral in \r{15} into two parts: over the set $\Sigma=\{z;\; |z-w|\leqslant\delta/C^0\}$, for some $C^0>1$, and then over the complement of $\Sigma$. Since $|\Phi_\zeta|$ has a ($\delta$-dependent) positive lower bound for $z \in \Sigma$( for $\zeta$ real) and there is no real critical point for the function $\zeta \rightarrow \Phi$ in this set, we can estimate that part of integral. Using the estimate \r{12}, integration by parts $N$-times w.r.t.\ $\zeta$, and the fact that on the boundary $|\zeta-\eta|=\delta$, the factor $e^{\i\lambda \Phi}$ is exponentially small with $\lambda$, we get
$$
	\Big|   \iint_{\Sigma^c}  e^{\i\lambda\Phi(z,w,\zeta,\eta)} b_N(z,\zeta,\eta )\tilde{f}_{i_v}(z)(1,\theta_{\pm}(\zeta))^i\,\d z\,
	\d \zeta     \Big| \le C(CN/\lambda)^N + CNe^{-\lambda/C}.
$$
	Note also that in the estimation above we used the fact that
$$
	e^{\i\lambda \Phi}=\frac{\bar{\Phi}_{\zeta}\cdot \partial_\zeta
	}{i\lambda |\Phi_\zeta|^2} e^{\i\lambda \Phi}.
$$
	Now on the set $\{z;\; |z-w|\leqslant\delta/\tilde{C}\}, \tilde{C}\gg1$, we apply the complex stationary phase method for the rest of $\zeta$-integral in \r{15}. To estimate \r{15} for $z\in \Sigma,$ we set: $\psi(z,w,\eta) = \Phi|_{\zeta=\zeta_c}.$
	Therefore,
	\[
	\psi = \eta\cdot (z-w) + \i |z-w|^2 -\frac{\i}2|z-w|^2 =
	\eta\cdot (z-w)+\frac{\i}2 |z-w|^2 .
	\]
	Clearly the new phase function $\psi(z,w,\eta)$ satisfies
$$
	\psi_z(z,z,\zeta)=\zeta, \quad \psi_w(z,z,\zeta)=-\zeta, \quad \psi(z,z,\zeta)=0.
$$
	For $(z,\zeta)$ close to $(0,\e^{1})$, we use this phase function and apply [Theorem~2.8, \cite{38}] and the remark after it to the $\zeta$-integral above to get
$$
	\iint_{\Sigma}  e^{\i\lambda\Phi_\mp(z,w,\zeta,\eta)} b_N(z,\zeta,\eta )\tilde{f}_{i_v}(z)(1,\theta_{\pm}(\zeta))^i\,\d z\,
	\d \zeta
$$
$$
	= \int_{\Sigma}  e^{\i\lambda\Phi(z,w,\zeta_c,\eta)} b_N(z,\zeta_c,\eta )\tilde{f}_{i_v}(z)(1,\theta_\pm(\zeta_c))^i\d z
$$
$$
	=
	\int_{\Sigma}  e^{\i\lambda\psi(z,w,\eta)} b_\lambda(z,w,\eta)\tilde{f}_{i_v}(z)(1,\theta_\pm(z,w,\eta))^i\d z
$$
$$
	= \int_{\Sigma}  e^{\i\lambda\psi(z,\beta)} b_\lambda(z,\beta)\tilde{f}_{i_v}(z)(1,\theta_\pm(z,\beta))^i\d z
$$
\be{16}
	=
	\int_{\Sigma}  e^{\i\lambda\psi(z,\beta)}\tilde{f}_{i_v}(z) \tilde{B}_{\lambda_\pm} ^i(z,\beta)\d z
	= \mathcal{O} (\lambda^{n/2}(CN/\lambda)^N + CNe^{-\lambda/C})
\ee
	where $\beta=(w,\eta)$, and $\tilde{B}_{\lambda_\pm}$ is a classical elliptic analytic symbol of order $0$ with principal part equal
$$
	\sigma_p(\tilde{B}_{\lambda_\pm}(z,z,\zeta))\equiv (1,\theta_\pm(\zeta)), \quad \text{ up to an elliptic factor near $(z,\beta)= (0,0,\e^{1})$},
$$
	with $\theta_\pm(\zeta)=(\theta_{+} ^1(\zeta),\theta_{+} ^2(\zeta))$ defined by \r{5}. In particular, for $(z,w,\zeta)=(0,0,\e^1)$ we have
$$
	\sigma_p(\tilde{B}_{\lambda_\pm}(0,0,\e^1))\equiv (1,\theta_\pm(\e^1))=(1,0,\pm 1)=(1,\pm \e_2), \quad \text{ up to an elliptic factor}.
$$
	For $z\in \Sigma$ with $\delta\ll1$ and $|w|\ll1$, $\eta$ close to $\e^{1}$, the variable $(z,\beta)$ in \r{16} is near $(0,0,\e^{1})$ and therefore $\tilde{B}_{\lambda_\pm}$is independent of $N$ as $\chi_N=1$ near the origin. Choose now $N$ so that  $N\le \lambda/(Ce)\le N+1$ to get the following exponential error on the right,
\be{17}
	\int_{\Sigma}  e^{\i\lambda\psi(z,\beta)} \tilde{f}_{i_v}(z) \tilde{B}_{\lambda_\pm} ^i(z,\beta)\d z = \mathcal{O} (e^{-\lambda/C}).
\ee

\textbf{Microlocal Ellipticity.}
	Now we show that for $(1,\theta_\pm(\zeta))$, equations in 	\r{17} form an elliptic system of equations at $(0,0,\zeta^0).$ Let $(z,z,\zeta)$ near $(0,0,\zeta^0)$ and $v$ be fixed, and consider the principal symbols $\sigma_p(\tilde{B}_{\lambda_\pm} (z,z,\zeta))\equiv (1,\theta_\pm(\zeta))$. Microlocal version of ellipticity is equivalent to show that for a constant vector field $\tilde{f}_v=(\tilde{f}_{0_v},\tilde{f}_{1_v},\tilde{f}_{2_v}),$
$$
	(1,\theta_\pm(\zeta))^i \tilde{f}_{i_v}=0
$$
	forms an elliptic system. Above equations can be written as
$$
	\left\{\begin{array}{ll}
	\tilde{f}_{0_v}+\frac{1}{|\xi|^2}(-\tau\xi^1 + \sqrt{-\tau^2+|\xi|^2}\xi^2) \tilde{f}_{1_v}+\frac{1}{|\xi|^2}(-\tau\xi^2 - \sqrt{-\tau^2+|\xi|^2}\xi^1) \tilde{f}_{2_v}=0,\\
	\tilde{f}_{0_v}+\frac{1}{|\xi|^2}(-\tau\xi^1 - \sqrt{-\tau^2+|\xi|^2}\xi^2) \tilde{f}_{1_v}+\frac{1}{|\xi|^2}(-\tau\xi^2 + \sqrt{-\tau^2+|\xi|^2}\xi^1) \tilde{f}_{2_v}=0.\\
	\end{array}
	\right.\
$$
	By similar arguments as it is shown in Proposition 3.1 for $n=2$, one may conclude that
$$
	\left\{\begin{array}{ll}
	0=\xi^2 \tilde{f}_{1_v}-\xi^1 \tilde{f}_{2_v}=(\xi \cdot v) (\partial_2 {f}_1 - \partial_1 {f}_2) \\
	0 \ =\tau \tilde{f}_{1_v}-\xi^1 \tilde{f}_{0_v}=(\xi \cdot v) (\partial_0 {f}_1 - \partial_1 {f}_0) \\
	\end{array}
	\right.\ \Longrightarrow \ \partial_2 {f}_1 - \partial_1 {f}_2=\partial_0 {f}_1 - \partial_1 {f}_0=0.
$$
	Clearly $\partial_2 {f}_1 - \partial_1 {f}_2=0$ implies that $\tilde{f}_{1_v}=\tilde{f}_{2_v}=0$ (for definition of $\tilde{f}_{i_v}$ see Proposition 2.1) and therefore
$$
	\tilde{f}_{0_v}=v^1(\partial_1 f_0 - \partial_0 f_1) + v^2(\partial_2 f_0- \partial_0f_2)=0.
$$
	Since $v$ is arbitrary in $\R^2$, any two linearly independent vectors $v_1$ and $v_2$ implies that $\partial_2 f_0- \partial_0f_2=0$. Therefore, the equation \r{17} leads to the following system of equations
$$
	\left\{\begin{array}{ll}
	\int_{\Sigma}  e^{\i\lambda\psi(z,\beta)} [v^1(\partial_1f_0-\partial_0f_1)(x) + v^2(\partial_2f_0-\partial_0f_2)(x)] B_{\lambda} ^0 (z,\beta)\,\d z\,
	= \mathcal{O} (e^{-\lambda/C})\\
	\int_{\Sigma}  e^{\i\lambda\psi(z,\beta)} [v^2 (\partial_2f_1-\partial_1f_2)] B_{\lambda} ^1 (z,\beta)\,\d z\,
	= \mathcal{O} (e^{-\lambda/C}) \\
	\int_{\Sigma}  e^{\i\lambda\psi(z,\beta)} [v^1 (\partial_1f_2-\partial_2f_1)] B_{\lambda} ^2 (z,\beta)\,\d z\,
	= \mathcal{O} (e^{-\lambda/C}), \\
	\end{array}
	\right.\
$$
	where up to an elliptic factor, we have
$$
	\sigma_p(B_{\lambda}^i(z,z,\zeta))\equiv \left\{\begin{array}{ll}
	\ \ 1 \quad \ \ \  \ \text{$i=0$} \\
	\theta_{+} ^1(\zeta) \quad \text{$i=1$} \\
	\theta_{+} ^2(\zeta) \quad \text{$i=2$.}\\
	\end{array}
	\right.
$$
	The vector $v$ is arbitrary in $\R^2$. Thus, for any choice of two linearly independent vectors, at $(z,z,\zeta)=(0,0,\zeta^0)$ above elliptic system of equations implies that $(0,\zeta^0)\not \in \WFA(\curl f)$ as desired. Notice that above system is an overdetermined system of equations since the term $\partial_1f_2-\partial_2f_1$ is repeated in the second and third equations. This is due to the property of $(v.\nabla_x)Lf$ and $\tilde{f}_v$ as we pointed out on Remark 2.2.

	Now if $f\in\mathcal{E}^\prime(\R^{1+2})$ is a distribution, as stated in the Theorem 4.1, the result still holds in the sense of distributions. In fact, one may take a sequence of $\mathcal{C}^1$-smooth functions which converges to the distribution $f$. The equation \r{14} holds for each smooth function. Now the $z$-integral in \r{14} can be thought in the sense of distributions as the integrand can be considered as the action of a distribution on a smooth function.

	ii) Let first $f\in \mathcal{C}^1(\R^{1+n}).$ By assumption, $Lf=0$ only near $\gamma_{x_0,\theta_0}$, so we choose a local chart for the lines close to $\gamma_{x_0,\theta_0}$. Since $Lf=0$, we have $(v.\nabla_x)Lf=0$ for $v \in \R^n.$  Let $v$ be an arbitrary fixed vector in $\R^n$ and let $x_0=0$ and $\theta_0=\pm \e_n$. Our goal is to show $(0,\zeta^0)\not \in \WFA(\d f)$ for $\zeta^0$ a non-zero space-like vector and conormal to $\gamma_0$ at $x_0=0$. Applying the Lorentz transformation, we may assume that $\zeta^0=\e^{n-1}:=(0,\dots,0,1,0)\in \R^{1+n}$. Let $\chi_N\in C_0^\infty(\R^n)$ be supported in $B(0,\eps)$, with $\eps>0$ and $\chi_N=1$ near $x_0=0$ so that
\be{18}
	|\partial_{x}^\alpha\chi_N|\le (CN)^{|\alpha|}, \quad \text{for $|\alpha|\le N$}.
\ee
	Then for $0<\eps\ll1$, $\lambda>0$, and $\theta$ near $\theta_0$,
	\[
	0= \int e^{\i \lambda x\cdot\xi} (\chi_N (v.\nabla_x)Lf)(x,\theta)\,\d x=
	\iint e^{\i\lambda x\cdot\xi} \chi_N (x)  \tilde{f}_v(\gamma_{x,\theta}(s))\cdot(1,\theta) \ \d s  \ \d x.
	\]
	Similar to the first part of theorem, we make a change of variable $z=\gamma_{x,\theta}$ to get
\be{19}
	0= \int e^{\i \lambda x\cdot\xi} (\chi_N (v.\nabla_x)Lf)(x,\theta) \d x=
	\int e^{\i\lambda x(z,\theta)\cdot\xi} a_N (z,\theta) \tilde{f}_{i_v}(z)(1,\theta)^i \d z
\ee
$$	
	=\int e^{\i\lambda z\cdot\zeta} a_N (z,\theta) \tilde{f}_{i_v}(z)(1,\theta)^i \d z,
$$
	when $(1,\theta)\bot \zeta$. Notice that $a_N(0,\theta)=1$.

	Let $a_1, a_2, \dots, a_{n-1},$ be $n-1$ non-zero parameters near zero. We set $\theta(a_1, a_2, \dots,\break a_{n-1})$ to be the $n$-dimensional spherical coordinates where
$$
	\left\{\begin{array}{ll}
	\theta^1=\sin (a_{n-1}) \sin (a_{n-2}) \sin (a_{n-3}) \dots \sin (a_4) \sin (a_3) \sin (a_2) \sin (a_1) \\
	\theta^2=\sin (a_{n-1}) \sin (a_{n-2}) \sin (a_{n-3}) \dots \sin (a_4) \sin (a_3) \sin (a_2) \cos (a_1)\\
	\theta^3=\sin (a_{n-1}) \sin (a_{n-2}) \sin (a_{n-3}) \dots \sin (a_4) \sin (a_3) \cos (a_2)\\
	\theta^4=\sin (a_{n-1}) \sin (a_{n-2}) \sin (a_{n-3}) \dots \sin (a_4) \cos (a_3)\\
	\vdots \\
	\theta^{n-2}=\sin (a_{n-1}) \sin (a_{n-2})\cos (a_{n-3})\\
	\theta^{n-1}=\sin (a_{n-1})\cos (a_{n-2})\\
 	\theta^{n}=\cos (a_{n-1})\\
 	\end{array}
	\right.\
$$
	Clearly $\big|\theta(a_1, a_2, \dots, a_{n-1})\big|=1$, $\theta(a_1, a_2, \dots, 0)=\e_n$. Considering the $n$-dimen\-sional spherical coordinate, one may solve the equation $(1,\theta)\cdot\zeta=0$ for $\zeta=(\tau,\xi)$ to get
$$
	\zeta((a_1, a_2, \dots, a_{n-1}),\xi)=(-\theta(a_1, a_2, \dots, a_{n-1})\cdot\xi,\xi).
$$
	To simplify our analysis, we show the rest of proof for $n=3$. For $n>3$, one may repeat the following arguments to conclude the result. Let
$$
	\theta(a,b)= \sin(a)\sin(b)\e_1 + \sin(a)\cos(b)\e_2+ \cos(a)\e_3
$$
	be the $3$-dimensional spherical coordinates. Plugging $\theta(a,b)$ into \r{19} we get
	$$
	\int e^{\i\lambda \phi(z,\zeta)} a_N (z,\theta(a,b))  \tilde{f}_{i_v}(z)(1,\theta(a,b))^i\,\d z=0, \quad \text{near $a=0$},
	$$
	where $\phi(z,\zeta)=z\cdot\zeta((a,b),\xi)$ and $a_N (z,\theta(a,b))=\chi_N(x-s\theta(a,b))$ with $a_N (0,(a,b))\break=1$. Note that
	$$
	\phi_z(0,\zeta)=\zeta, \quad \phi_{z\zeta}=\Id.
	$$
	It is more convenient to work with $\zeta$ variable instead of $((a,b),\xi)$. So let $b$ be a non-zero fixed parameter near zero. We show that the map $((a,b),\xi)\rightarrow \zeta \in \R^{1+3}$ is a local analytic diffeomorphism near $((0,b),\e^{2})$. More precisely, the determinant of Jacobean associated with the map $((a,b),\xi)\rightarrow \zeta \in \R^{1+3}$ is
$$
	-\cos(a)\sin(b)\xi^{1}-\cos(a)\cos(b)\xi^{2}+\sin(a)\xi^3,
$$
	which is equal to $-\sin(b)\xi^{1}-\cos(b)\xi^{2}$ near $a=0$. Now the fixed parameter $b$ (near zero) and our choice of $\zeta^0$ imply that the determinant is $-\cos(b)$ which is non-zero. Hence, one may apply the Implicit Function Theorem near $a=0$ to locally invert the map to $\zeta \rightarrow  ((a,b),\xi)\in \R^{1+3}$. One may compute $a$ explicitly to get
$$
	a=a(\zeta)=-\tan^{-1}(\frac{\xi^3}{\sin(b)\xi^{1}+\cos(b)\xi^{2}})+ \sin^{-1} (-\frac{\tau}{\sqrt{(\sin(b)\xi^{1}+\cos(b)\xi^{2})^2+(\xi^{3})^2}})
$$
	which maps $a=0$ to $\zeta^0$ diffeomorphically. Notice that for the fixed parameter $b$  and $\zeta \approx \zeta^0$, $a(\zeta)$ is the unique solution of the equation
$$
	-\tau= -\theta(a,b)\cdot\xi =\sin(a)\sin(b)\xi^{1}+\sin(a)\cos(b)\xi^{2}+\cos(a)\xi^3,
$$
	near $a=0$. Therefore, we may work in the $\zeta$ variables instead of the $((a,b),\xi)$ to get
$$
	\int e^{\i\lambda z\cdot\zeta} \tilde{a}_N (z,\zeta) \tilde{f}_{i_v}(z)(1,\theta(\zeta))^i\,\d z=0, \quad 	\text{near $\zeta=\e^{2}$},
$$
	where $\tilde{a}_N (z,\zeta)=a_N (z,\theta(a,b))$ and $\tilde{a}_N (0,\zeta)=1$.

	In the next step, we analyze the phase function and its critical points. (A similar argument as in the first part of theorem by applying the complex stationary phase method of Sj\"ostrand)

	Fix $0<\delta\ll1$ and let $\chi_\delta$ be the characteristic function of the unit ball $B(0,\delta)$ in $\mathbf{R}^{1+3}$. With some $w$, $\eta\in \mathbf{R}^{1+3}$ close to $w=0$, $\eta=\e^{2}$, multiply the l.h.s.\ of above integral equation by
$$
	\chi_\delta(\zeta-\eta) e^{ \i \lambda(\i ( \zeta-\eta)^2/2 -  w\cdot\zeta)}
$$
	and integrate w.r.t.\ $\zeta$ to get
\be{20}
	\iint e^{\i\lambda\Phi(z,w,\zeta,\eta)} b_N(z,\zeta,\eta)\tilde{f}_{i_v}(z)(1,\theta(\zeta))^i\d z\, \d \zeta=0, \quad  \text{near $\zeta=\e^{2},$}
\ee	where $b_N= \chi_\delta(\zeta-\eta)\tilde{a}_N$ is a new amplitude and
	\[
	\Phi = (z-w)\cdot\zeta +  \i (\zeta-\eta)^2/2.
	\]
	Now consider the phase function $\zeta \rightarrow \Phi$. If $w=z$, there is a unique real critical point $\zeta_c =\eta$, which satisfies $\Im \Phi_{\zeta \zeta}> 0$ at $\zeta=\zeta_c$. For $w\not=z$, the phase $\Phi$, as function of $\zeta$, has a unique critical point $\zeta_c =\eta +\i (z-w). $

	Now we split the $z$-integral \r{20} into two parts: over $\Sigma=\{z;\; |z-w|\leqslant\delta/C^0\}$, for some $C^0>1$, and then over the complement of $\Sigma$. Since $|\Phi_\zeta|$ has a ($\delta$-dependent) positive lower bound for $|z-w|>\delta/C^0$( for $\zeta$ real) and there is no real critical point for the function $\zeta \rightarrow \Phi$ in this set, we can estimate that part of integral.
	Using the estimate \r{18}, integration by parts $N$-times w.r.t.\ $\zeta$, and the fact that on the boundary $|\zeta-\eta|=\delta$, the factor $e^{\i\lambda \Phi}$ is exponentially small with $\lambda$, we get
$$
	\Big|   \iint_{\Sigma^c}  e^{\i\lambda\Phi(z,w,\zeta,\eta)} b_N(z,\zeta,\eta )\tilde{f}_{i_v}(z)(1,\theta(\zeta))^i \d z\,
	\d \zeta     \Big| \le C(CN/\lambda)^N + CNe^{-\lambda/C}.
$$
	Similar to part (i), for above inequality we used the fact that
$$
	e^{\i\lambda \Phi}=\frac{\bar{\Phi}_{\zeta}\cdot \partial_\zeta
	}{i\lambda |\Phi_\zeta|^2} e^{\i\lambda \Phi}.
$$
	Now on the set $\{z;\; |z-w|\leqslant\delta/\tilde{C}\}, \tilde{C}\gg1$, we apply stationary phase method for the rest of $\zeta$-integral in \r{20}. To estimate \r{20} for $z\in \Sigma$, we set: $\psi(z,w,\eta) = \Phi|_{\zeta=\zeta_c}.$
	Therefore,
\[  \quad
	\psi = \eta\cdot (z-w) + \i |z-w|^2 -\frac{\i}2|z-w|^2 =
	\eta\cdot (z-w)+\frac{\i}2 |z-w|^2 .
\]
	Notice that $\psi(z,w,\eta)$ satisfies
\be{21}
	\psi_z(z,z,\eta)=\eta=\phi_z(0,\eta), \quad \psi_w(z,z,\eta)=-\eta=-\phi_z(0,\eta), \quad \psi(z,z,\eta)=0.
\ee
	For $(z,\zeta)$ close to $(0,\e^{2}),$ we use this phase function and apply [Theorem~2.8, \cite{38}] and the remark after it to the $\zeta$-integral above to get
$$
	\iint_{\Sigma}  e^{\i\lambda\Phi_\mp(z,w,\zeta,\eta)} b_N(z,\zeta,\eta )\tilde{f}_{i_v}(z)(1,\theta(\zeta))^i\,\d z\,
	\d \zeta
$$
$$	
	= \int_{\Sigma}  e^{\i\lambda\Phi(z,w,\zeta_c,\eta)} b_N(z,\zeta_c,\eta )\tilde{f}_{i_v}(z)(1,\theta(\zeta_c))^i\d z
$$
$$
	=
	\int_{\Sigma}  e^{\i\lambda\psi(z,w,\eta)} b_\lambda(z,w,\eta)\tilde{f}_{i_v}(z)(1,\theta(z,w,\eta))^i\d z
$$
$$
	= \int_{\Sigma}  e^{\i\lambda\psi(z,\beta)} b_\lambda(z,\beta)\tilde{f}_{i_v}(z)(1,\theta(z,\beta))^i\d z
$$
\be{22}
	=
	\int_{\Sigma}  e^{\i\lambda\psi(z,\beta)}\tilde{f}_{i_v}(z) \tilde{B}_{\lambda} ^i(z,\beta)\d z
	= \mathcal{O} (\lambda^{n/2}(CN/\lambda)^N + CNe^{-\lambda/C}).
\ee
	Here $\beta=(w,\eta)$ and $\tilde{B}_{\lambda}$ is a classical elliptic analytic symbol of order $0$.
	For $z\in \Sigma$ with $\delta\ll1$ and $|w|\ll1$, $\eta$ near $\e^{2}$, the variable $(z,\beta)$ in \r{22} is near $(0,0,\e^{2})$ and then $\tilde{B}_{\lambda}$ is independent of $N$ because $\chi_N=1$ near the origin. We choose $N$ so that $N\le \lambda/(Ce)\le N+1$. Therefore, we get the following exponential error on the right
$$
	\int_{\Sigma}  e^{\i\lambda\psi(z,\beta)} \tilde{f}_{i_v}(z)\tilde{B}_{\lambda} ^i(z,\beta)\d z\,
	= \mathcal{O} (e^{-\lambda/C}).
$$
	Since the phase function satisfies the properties in \r{21}, on a small neighborhood of $\zeta^0$, we perform the following change of variable in above integral equation:
$$
	(z,w,\eta)\longrightarrow(z,w,\zeta)=(z,w,\phi_z(w,\eta)),
$$
	which yields to
\be{23}
	\int_{\Sigma}  e^{\i\lambda\psi(z,w,\zeta)} \tilde{f}_{i_v}(z) \tilde{B}_{\lambda} ^i(z,w,\zeta)\d z
	= \mathcal{O} (e^{-\lambda/C}).
\ee
	Here $\tilde{B}_{\lambda}$ is a new classical elliptic symbol of order zero with the principal part of $\sigma_p(\tilde{B}_{\lambda}(z,z,\zeta))\equiv (1,\theta(\zeta)),$ up to an elliptic factor. In particular, for $(z,w,\zeta)=(0,0,\zeta^0)$ we have
$$
	\sigma_p(\tilde{B}_{\lambda}(0,0,\zeta^0))\equiv (1,\theta(\zeta^0))=(1,0,0,1)=(1,\e_3).
$$
	As it is shown above, the map $(a,\xi)\rightarrow \zeta$ is a local diffeomorphsim near $a=0$ (equivalently near $\zeta^0=\e^{2}$). Therefore, we work with the principal symbol in terms of $(a,\xi)$ instead, which means up to an elliptic factor
$$
	\sigma_p(\tilde{B}_{\lambda}(z,z,(a,\xi))\equiv (1,\theta(a,b)).
$$
	To show $(0,\zeta^0) \not \in \WFA(\d f)$, we need to form an elliptic system of equations using \r{23}. Let $(z,z,(a,\xi))\approx(0,0,(0,\xi^0))$ and $v$ be a fixed vector. For our goal, we slightly perturb $\theta(a,b)$ near $a\approx 0$ and $b$. Let
	\[
	\{\Theta_k\}_{k=0} ^3
	=\lbrace(1,\theta(a,b)),(1,\theta(-a,b)),(1,\theta(a,-b)),(1,\theta(0,b))\rbrace
	\]
	be the set of perturbations of $\theta(a,b)$, with property $\sigma_p(\tilde{B}_{\lambda_k}(z,z,(a,\xi))=\Theta_k$, for $k=0,1,2,3$.
	Microlocal version of ellipticity is equivalent to show that for a constant vector field $\tilde{f}_v=(\tilde{f}_{0_v},\tilde{f}_{1_v},\tilde{f}_{2_v},\tilde{f}_{3_v}),$
	$$
	\Theta \tilde{f}_v=0
	$$
	forms an elliptic system of equations. Here the matrix $[\Theta]_{4\times4}$ is the associated matrix with above principal symbols, $\Theta_k.$ The matrix $[\Theta]$ is invertible since its determinant equals to
{\small$$
	\det \begin{pmatrix}
	1&\sin(a)\sin(b)&\sin(a)\cos(b)&\cos(a)\\
	1&-\sin(a)\sin(b)&-\sin(a)\cos(b)&\cos(a)\\
	1&-\sin(a)\sin(b)&\sin(a)\cos(b)&\cos(a)\\
	1&0&0&1
	\end{pmatrix}
	= 4\sin^2(a)\sin(b)\cos(b)(1-\cos(a))
$$}which is non-zero for $a$ and our fixed parameter $b$ near zero. Therefore, $ \Theta\tilde{f}_v=0$ implies that $\tilde{f}_v=0.$ This means the equation \r{23} with $\{\Theta_k\}_{k=0} ^3$ leads to the following system of equations related to $\theta_0=\e_3$:
$$
	\left\{\begin{array}{ll}
	\int_{\Sigma}  e^{\i\lambda\psi(z,\beta)} \tilde{f}_{0_v} B_{\lambda} ^0 (z,\beta)\d z\,
	= \mathcal{O} (e^{-\lambda/C})\\
	\int_{\Sigma}  e^{\i\lambda\psi(z,\beta)} \tilde{f}_{1_v} B_{\lambda} ^1 (z,\beta)\d z\,
	= \mathcal{O} (e^{-\lambda/C}) \\
	\int_{\Sigma}  e^{\i\lambda\psi(z,\beta)} \tilde{f}_{2_v} B_{\lambda} ^2 (z,\beta)\d z\,
	= \mathcal{O} (e^{-\lambda/C}) \\
	\int_{\Sigma}  e^{\i\lambda\psi(z,\beta)} \tilde{f}_{3_v} B_{\lambda} ^3 (z,\beta)\d z\,
	= \mathcal{O} (e^{-\lambda/C}), \\
	\end{array}
	\right.\
$$
	where up to an elliptic factor, we get
	$$
	\sigma_p(B_{\lambda}^i(z,z,\zeta))\equiv \left\{\begin{array}{ll}
	1 \quad \quad \quad \quad \quad \ \ \text{$i=0$} \\
	\sin(a)\sin(b) \quad \text{$i=1$} \\
	\sin(a)\cos(b) \quad \text{$i=2$}\\
	\cos (a) \quad \quad \quad \ \ \text{$i=3$}. \\
	\end{array}
	\right.
$$
	Here $\tilde{f}_{i_v}$ is defined by Proposition 2.1. Also, the phase function $\psi$ satisfies the conditions \r{21} and $\Im \psi>C_0|z-w|^2$ as $\psi_{z\zeta}=\Id$. Note that, some components of $\d f$ are repeated in above equations. This forms an overdetermined system of equations for our fixed vector $v \in \R^3$. Since $v$ is arbitrary, for any choice of three linearly independent vectors $\{v_i\}_{i=1} ^3 \subseteq \R^3$, for  $(z,z,(z,\xi))=(0,0,(0,\xi^0))$ one may conclude that $(0,\zeta^0)\not \in \WF_A(\d f)$, which proves the second part of the theorem  for $n=3$. For the general case $n>3$, one needs to slightly perturb $\theta(a_1, a_2, \dots, a_{n-1})$ with respect to parameters $a_1, a_2, \dots, a_{n-1}$. This forms an elliptic system $\Theta \tilde{f}_v=0$ for the microlocal ellipticity discussion and therefore concludes the result. Now for any vector-valued distribution $f\in\mathcal{E}^\prime(\R^{1+n})$, as we pointed out in the proof of part $(i)$, the result remains true in the sense of distributions.
\end{proof}		

\begin{remark}
	\normalfont
	By Fundamental Theorem of Calculus, the potential field is in the kernel of operator $L$, so one could only hope to recover $\curl f$ for $n=2$. For the Riemannian case with dimension $n\ge 3$, foliation (slicing method) can be used to achieve the uniqueness results. One may restrict $x$ to a two-dimensional plane, say $\Pi =\{(t,x): x^3=\dots=x^n=\const\}$, and apply the results in Theorem 4.1 when $n=2$. This only recovers some components of the generalized $\curl$ of the vector field $f$ even if different permutations are chosen to fix different components of $x$. In order to recover all other components, one needs to perturb above two-dimensional planes. Therefore, such a slicing technique can be done as the transform is overdetermined. However, additional assumption which is the information of light-ray for two discrete directions $(1,\pm\theta)$ is required. Even though the foliation method is a simpler approach for the recovery of the vector field $f$, we do not perform foliation to achieve stronger results.
\end{remark}

\section{Proof of main results}
	For our main result we need the following lemma which is a unique analytic continuation result across a time-like hypersurface in the Minkowski time-space.

\begin{lemma}
	Let $f \in \mathcal{C}^\infty (\R^{1+n})$ and let $\gamma_{x_0,\theta_0}$ be a fixed light-like line in the Minkowski time-space so that $\gamma_{x,\theta}$ does not intersect $\supp f$  for $|s| \ge 1/C$ with some C for all $(x,\theta)$ near $(x_0,\theta_0)$. Fix $z_0 = (s_0 ,x_0 + s_0 \theta_0) \in \gamma_{x_0,\theta_0}$, let $S$ be an analytic time-like hypersurface near $z_0$ and assume that $\gamma_{x_0,\theta_0}$ is tangent to $S$ at $z_0$.
	
	i) For $n=2$, if $L f(x,\theta) = 0$ near $(x_0,\pm\theta_0)$ and if $\curl f=0$ on one side of $S$ near $z_0$, then $\curl f=0$ near $z_0$.
	
	ii) For $n\ge 3$, if $L f(x,\theta) = 0$ near $(x_0,\theta_0)$ and if $\d f=0$ on one side of $S$ near $z_0$, then $\d f=0$ near $z_0$.
\end{lemma}

\begin{proof}
	Let $n \ge 3$ and assume that $z_0 \in \supp \d f$. By assumption $(v.\nabla_x)Lf(x,\theta)=0$ near $(z_0,\theta_0)$. Since $\d f$ is non-zero only in half space $S$, then there exists $\zeta_0$ such that $(z_0,\zeta_0) \in \WFA (\d f)$, as $\d f$ cannot be analytic at $z_0$. By the definition of analytic wave front set for vector-valued distrubitions, there exist a component of $\d f$, say $f_{ij}=\partial_jf_i-\partial_if_j$, where $(z_0,\zeta_0) \in \WFA (f_{ij})$. In other words, if the half space $S$ intersects the $\supp \d f$, it must intersect at least one of the components of $\d f$, say $f_{ij}$, as $\d f$ cannot be analytic at the intersection point. Now by Sato-Kawai-Kashiwara Theorem (see \cite{36, 38}), $(z_0 ,\zeta_0 + s \nu(z_0)) \in  \WFA (f_{ij})$, where $\nu(z_0)$ is one of the two unit conormals to $S$ at $z_0$. This in turn implies that $(z_0 ,\frac{\zeta_0}{s} + \nu(z_0)) \in  \WFA (f_{ij})$ as the wave front set is a conic set. Now by passing to limit, we have $(z_0 , \nu(z_0)) \in  \WFA (f_{ij})$ since the analytic wave front set is closed. By assumption on $S$, that vector is space-like and is conormal to $\gamma^{\prime˙} _{x_0,\theta_0} (s_0)$. This contradicts Theorem 4.1 part $(ii)$, which implies that $\d f=0$ and completes the proof. For $n=2$, one may repeat above arguments and use the first part of Theorem 4.1 to conclude the result.
\end{proof}
	We now are ready to state the proof of the main result.
\begin{proof}[\textit{Proof of Theorem 2.1.}]
	Let $n\ge 3$. By assumption $(v.\nabla_x)Lf(x,\theta)=0$. The proof follows from [Theorem~2.1, \cite{39}] replacing $Lf(x,\theta)$ by $(v.\nabla_x)Lf(x,\theta)$ and applying the second part of Lemma 5.1.
 
 To conclude the result for $n=2$, one may use Lemma 5.1 part $(i)$ and repeat the proof of [Theorem~2.1, \cite{39}].
\end{proof}

\section{Examples}
	In the following examples we illustrate how our result imply the recovery of vector field up to a smooth potential field.
\begin{example}
	\normalfont
	Let $f$ be a vector field (distribution) supported in cone $\lbrace (t,x)\in \R^{1+n} | \ |x|<c|t|\rbrace$ and let $\Gamma_{\rho_0}$ be the following surface:
$$
	\Gamma_{\rho_0}= \lbrace(t,x)\in \R^{1+n} \ | \ \psi(t,x)=|x-x_0|^2 - c^2|t-t_0|^2 - \rho^2 _0=0\rbrace,
$$
	for some $\rho_0 \geq 0 $ and  $0 < c< 1$. Assume now that $f$ integrates to zero over all light-like lines $\gamma$ in the exterior of $\Gamma_{\rho_0}$, $\ext (\Gamma_{\rho_0})$.

	We show the vector field $f$ can be recovered up to a potential field in the $\ext (\Gamma_{\rho_0})$. By definition $(z,\zeta)$ is conormal to $\Gamma_{\rho_0}$ if and only if
{\small$$
	(z,\zeta)\in N^*\Gamma_{\rho_0} = \lbrace(t,x,\tau,\xi)\in T^*(\R^{1+n} \times \R^{1+n})| (t,x)\in \Gamma_{\rho_0} , (\tau,\xi)=0 \text{ on $T_{(t,x)} \Gamma_{\rho_0}$} \rbrace.
$$}Clearly the gradient of $\psi$, $\nabla \psi$, is normal to surface $\Gamma_{\rho_0}$ at $(t,x)$. So by definition $\sigma \nabla \psi \cdot (dt, dx)$ is the conormal vector to surface $\Gamma_{\rho_0}$ at $(t,x)$. In fact, to find the conormal we compute the total differential of $\psi$, $d\psi$:
$$
	d\psi(t,x)= -2 c^2 (t-t_0) dt + 2 (x-x_0) dx,
$$
	and therefore the covector:
$$
	(z,\zeta)=(t,x,\tau,\xi) =(t,x,-2 \sigma c^2 (t-t_0), 2\sigma (x-x_0)) \in N^* \Gamma_{\rho_0}, \quad \text{ for $\sigma \in \R$},
$$
	is conormal to $\Gamma_{\rho_0}$. Clearly the $\zeta=(\tau,\xi)$ is space-like in the $\ext (\Gamma_{\rho_0})$ as it is easy to show $|\tau|= 2\sigma c^2|t-t_0|\leq 2\sigma|x-x_0|=|\xi|$ in the $\ext(\Gamma_{\rho_0})$. Therefore, for any decreasing family of $\rho$ with the property $\rho \rightarrow \rho_0$, the surfaces $\Gamma_\rho$ will be a family of analytic time-like hypersurfaces in the $\ext(\Gamma_{\rho_0})$.
	
	Let  $\sigma=\frac{1}{2}$, $(t_0,x_0)=(0,0)$, and fixed $\rho>0$ be the smallest one with the property that $\supp f \cap \Gamma_{\rho} \not = \emptyset$ ($\supp f\cap \Gamma_\rho$ is a compact set.)
	Now assume that $\gamma_0$ is tangent to $\Gamma_{\rho}$ at $z_0$ (i.e. $(z_0,\zeta_0)$ is conormal to $\dot{\gamma_0}$). By compactness of $\supp f$, we have $\gamma \not \in \supp f$ for any $\gamma$ approaching $\gamma_0$ in the $\ext(\Gamma_{\rho_0})$. Theorem 4.1 implies $\WFA(\d f)$ contains no space-like vector conormal to $\dot{\gamma_0}$ since by assumption $Lf=0$ over all light-like lines $\gamma$ near $\gamma_0$ on one side of analytic time-like hypersurface $\Gamma_{\rho}$, see Figure 1. Now using the analytic continuation result, Lemma 5.1, we can recover the vector field $f$ up to a potential field.
	\begin{figure}[H]
	\vspace*{-8pt}	\includegraphics[scale=0.3]{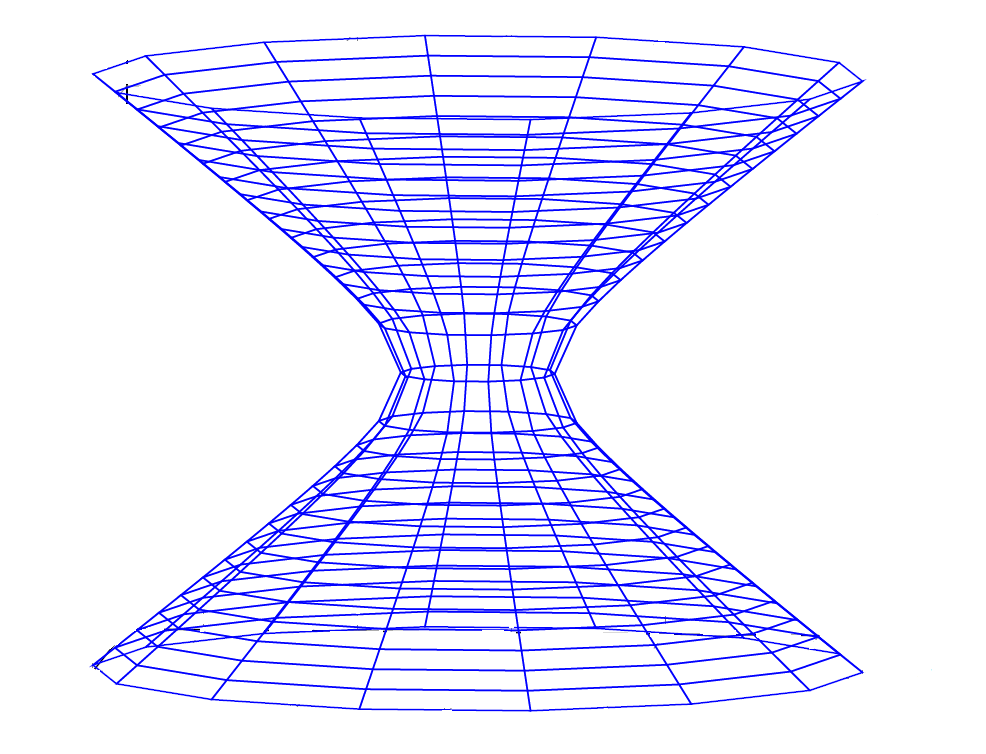}
		\caption{$\Gamma_{\rho_0}$ with $0<c< 1$.}\vspace*{-8pt}
	\end{figure}
\end{example}

\begin{example}
	\normalfont
	Let $f \in \mathcal{E}'(\R^{1+n})$ be so that $\supp f$ expands with a speed less than one and let $\Gamma_0$ be the following surface:
$$
	\Gamma_0 = \lbrace(t,x)\in \R^{1+n} \ | \ \psi(t,x)=(|x-x_0|-R)^2 - c^2|t-t_0|^2 =0 \rbrace, \quad \text{for some $0<c<1$.}
$$
	Assume now that $f$ integrates to zero over all light-like lines $\gamma$ intersecting $\supp f$ outside of the surface $\Gamma_0$. We show the vector field $f$ can be recovered up to a potential field in the $\ext (\Gamma_{\rho})$, where the surface $\Gamma_\rho$ with $\rho>0$ is defined as follow:
$$
	\Gamma _\rho = \lbrace(t,x)\in \R^{1+n} \ | \ \psi(t,x)=(|x-x_0|-R)^2 - c^2|t-t_0|^2 - \rho^2=0 \rbrace.
$$
	By definition, the covector:
$$
	(z,\zeta)=(t,x,\tau,\xi) =(t,x,-2 \sigma c^2 (t-t_0), 2\sigma \frac{x-x_0}{|x-x_0|}(|x-x_0|-R)) \in N^* \Gamma_\rho,\text{for $\sigma \in \R$},
$$
	is conormal to $\Gamma_\rho$. Note that $|\tau|= 2\sigma c^2|t-t_0|$ and $|\xi|= |2\sigma \frac{x-x_0}{|x-x_0|}(|x-x_0|-R)|=2\sigma ||x-x_0|-R|$. So in the $\ext(\Gamma_\rho)$, the covector $\zeta=(\tau,\xi)$ is space-like (i.e. $|\tau| \leq |\xi|$.) Thus, for $\rho>0$, the surfaces $\Gamma_\rho$ will be a family of analytic time-like hypersurfaces.
	
	Let $\sigma=\frac{1}{2}$, $z_0=(t_0,x_0)=(0,0)$, and $\rho>0$ fixed be the smallest one with the property that $\supp f \cap \Gamma_{\rho}\not = \emptyset$. Similar argument as in above example shows that the vector field $f$ can be recovered up to a potential field on $\Gamma_\rho$. The case where $c=0$ corresponds to the classical support theorem for balls.
	\begin{figure}[H]
	\vspace*{5pt}	\includegraphics[scale=0.32]{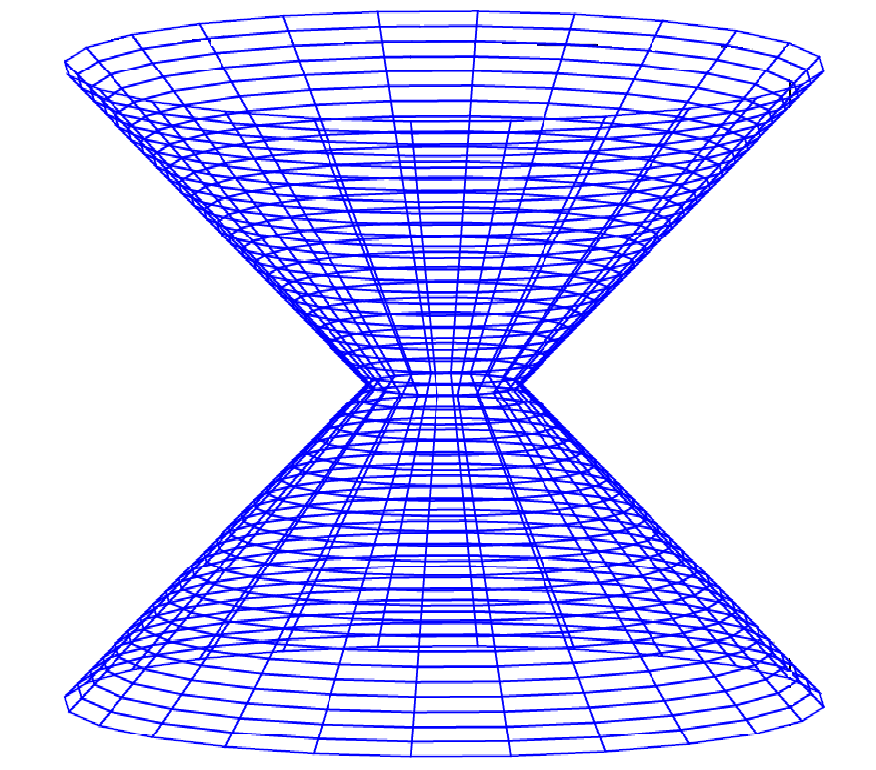}
		\caption{$\Gamma_{0}$ with $0<c<1$.}\vspace*{-8pt}	
	\end{figure}
	Note that on the surface $\Gamma_0$, there is no conormal covector at $t=t_0$ as $|x-x_0|=R$. Being outside of $\Gamma_\rho$ guarantees the existence of conormal covector as $\rho>0$ on $\Gamma_\rho$, see Figure 2.
	\end{example}

	Next example is a partial data case of Example 6.2 for the inverse recovery of a smooth potential field for the hyperbolic Dirichlet-to-Neumann (DN) maps. It is known that, all the integral lines can be extracted from the DN map for hyperbolic (wave) equations, see, e.g., \cite{40, 32, 31, 49, 33, 4, 3}. Our result provides the optimal way of the inverse recovery process up to a smooth potential. A similar result for recovery of the unknown potential can be found in \cite{46}.
\begin{example}
	\normalfont
	Let $f \in \mathcal{E}'(\R^{1+n})$ be so that $\supp f$ expands with a speed less than one, and consider the cylinder $[0,T]\times \bar{\Omega}$ for some $T>0$ and $\Omega \subset \R^n$. In Example 6.2 we showed that the vector field $f$ can be recovered up to a smooth potential field in the exterior of $\Gamma_\rho$. Now consider the surface $\Gamma$ as union of all those exteriors of time-like hypersurfaces for $t\in[0,T]$. This surface includes all light-like lines $\gamma_{x,\theta}= (s, x + s\theta)$, $(z,\theta) \in \R^n \times S^{n-1}$, not intersecting the top and the bottom of the cylinder $[0,T]\times \bar{\Omega}$, see Figure 3.
		\begin{figure}
		\includegraphics[scale=0.45]{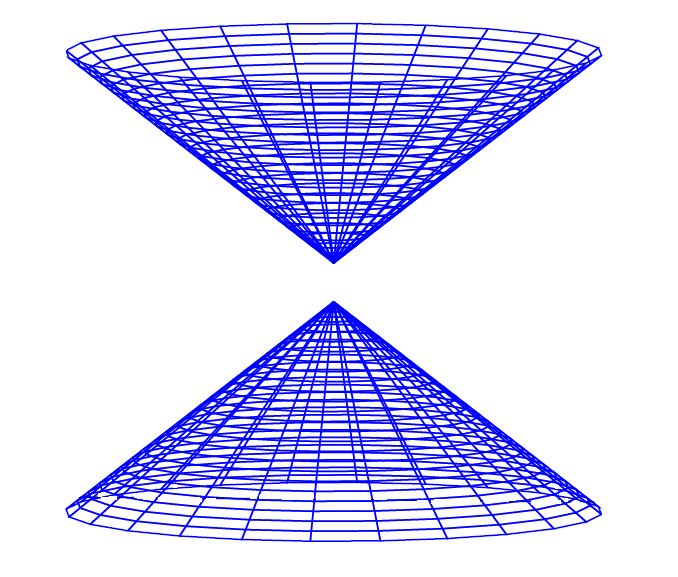}
		\caption{$\Gamma_{0}$ with $0<c< 1$.}
	\end{figure}

	By Theorem 2.1, we can recover $f$ up to $d\phi$, $\phi = 0$ on $[0,T]\times \bar{\Omega}$ in the set covered by those lines for $n\geq3$. As we pointed out on Remark 3.1, for $n = 2$ there are two discrete choices of directions which means that for recovery of $f$ up to a potential $d\phi$ one needs to know $Lf$ along light-like $\gamma_{x,\theta}$ as well as knowing $Lf$ along $\gamma_{x,-\theta}$. Note that our uniqueness results do not require the vector field to be compactly supported in time. Moreover, we are not considering any Cauchy data on circles on top and bottom of the cylinder, which means there is no internal measurement. This is the optimal way one can wish to recover the vector field $f$ in this set. Uniqueness result in this paper and result in \cite{46} (recovery of the unknown potential $q$ in $\Gamma$) generalize the uniqueness results in \cite{40, 32, 31, 49, 33, 4, 3}.
\end{example}

\section*{Acknowledgment.}
	The author would like to express his gratitude for Professor Plamen D Stefanov for suggesting the problem and many useful discussions.

\medskip
Received October  2016; revised December  2017.
\medskip

{\it E-mail address: }srabieni@purdue.edu\\

\end{document}